\documentclass[12pt,twoside]{amsart}
\usepackage{amssymb}
\usepackage{amscd}
\usepackage[dvips]{graphics}

\title[Smooth projective 
toric threefolds]{Smooth projective toric varieties whose nontrivial 
nef line bundles are big}
\author{Osamu Fujino and Hiroshi Sato} 
\date{2008/10/23}
\newcommand{\Pic}[0]{{\operatorname{Pic}}}
\newcommand{\Nef}[0]{{\operatorname{Nef}}}
\newcommand{\PE}[0]{{\operatorname{PE}}}
\newcommand{\Amp}[0]{{\operatorname{Amp}}}
\newcommand{\NE}[0]{{\operatorname{NE}}}
\newcommand{\xBig}[0]{{\operatorname{Big}}}

\newtheorem{thm}{Theorem}[section]
\newtheorem{lem}[thm]{Lemma}
\newtheorem{cor}[thm]{Corollary}
\newtheorem{prop}[thm]{Proposition}
\newtheorem*{claim}{Claim}
\newtheorem{ques}[thm]{Question}

\theoremstyle{definition}
\newtheorem{ex}[thm]{Example}
\newtheorem{defn}[thm]{Definition}
\newtheorem{rem}[thm]{Remark}
\newtheorem*{ack}{Acknowledgments}      
\newtheorem*{notation}{Notation}        
\newtheorem{say}[thm]{}
\begin{document}
\bibliographystyle{amsalpha+}

\begin{abstract} 
For any $n\geq 3$, 
we explicitly construct smooth 
projective toric $n$-folds 
of Picard number $\geq 5$, where any 
nontrivial nef line bundles are big. 
\end{abstract}
\maketitle
\tableofcontents

\section{Introduction}

The following question is our main motivation of 
this note. 

\begin{ques}
Are there any smooth projective toric 
varieties $X\not \simeq \mathbb P^n$ such that 
$$
\partial \Nef (X) \cap \partial \PE(X)=\{0\}\ ?   
$$ 
Here, $\Nef(X)$ is the nef cone of $X$ and 
$\PE(X)$ is the pseudo-effective cone of $X$.  
\end{ques}

By definition, the nef cone 
$\Nef (X)$ is included in the pseudo-effective cone 
$\PE(X)$. 
We note that 
$
\partial \Nef (X) \cap \partial \PE(X)=\{0\}$ 
is equivalent to the condition that any nontrivial nef line bundles on $X$ are big when $X\not\simeq \mathbb P^n$. 

In this note, we explicitly construct smooth {\em{projective}} 
toric threefolds of 
Picard number $\geq 5$ on which 
any nontrivial nef line bundles are big. 
The main 
parts of this note are nontrivial examples given in Section \ref{sec4}. 
See Examples \ref{321} and \ref{43}. 
In general, 
it seems to be hard 
to find those examples. 
Therefore, it must be valuable to describe them explicitly here. 
This short note is a continuation and 
a supplement of the papers:~\cite{kleiman} 
and \cite{smooth}. 

Let us see the contents of this note. 
Section \ref{sec2} is a supplement to the toric Mori theory. 
We introduce the notion of \lq general\rq \ complete 
toric varieties. 
By the definition of \lq general\rq \ projective toric varieties, 
it is obvious that the final step of the MMP for a 
$\mathbb Q$-factorial \lq 
general\rq \ projective toric variety 
is a $\mathbb Q$-factorial projective toric variety of Picard 
number one. It is almost obvious if we understand Reid's 
combinatorial description of 
toric extremal contraction morphisms. 
Moreover, it is easy to check that any nontrivial nef line bundles 
on a \lq general\rq \ complete toric variety are always big. 
In Section \ref{sec3}, we recall the basic definitions and 
properties of {\em{primitive collections}} and {\em{primitive 
relations}} after Batyrev. 
By the result of Batyrev, any {\em smooth} projective 
toric variety is \lq general\rq \  if and only if it is isomorphic 
to the projective space. 
So, the results obtained in Section \ref{sec2} can not be used 
to construct examples in Section \ref{sec4}. 
The first author first considered that there 
are plenty of \lq general\rq \ smooth projective toric varieties. 
So, he thought that the examples in Section \ref{sec4} 
is worthless. 
Section \ref{sec4} is the main part of this note. 
We give smooth projective toric threefolds of 
Picard number $\geq 5$, where any nontrivial 
nef line bundles are always big. We note that this phenomenon 
does not occur for smooth projective toric surfaces. 
Let $X$ be a smooth projective toric surface. 
Then we can easily see that there exists a morphism 
$f:X\to \mathbb P^1$ if $X$ is not isomorphic to $\mathbb P^2$. 
So, the line bundle $f^*\mathcal O_{\mathbb P^1}(1)$ 
on $X$ is nef but not big. 
Let $X$ be a smooth projective toric variety and 
let $\Delta$ be the corresponding fan. 
If $\Delta$ is sufficiently complicated combinatorially in some sense, 
then any nontrivial nef line bundles are big. 
However, we do not know how to define \lq complicated\rq \ 
fans suitably. Therefore, the 
explicit examples in Section \ref{sec4} 
seem to be useful. We note that 
it is difficult to calculate nef cones or pseudo-effective 
cones for projective (not necessarily toric) varieties. 
In the final section:~Section \ref{sec5}, 
we collect miscellaneous results. 
We explain how to generalize examples in \cite{smooth} 
and 
in Section \ref{sec4} into dimension $n\geq 4$. 
We also treat $\mathbb Q$-factorial 
projective toric varieties 
with $\Nef(X)=\PE(X)$. 

Let us fix the notation used in this note. 
For the details, see \cite{reid} or \cite{intro}. 
For the basic results on the toric geometry, 
see the standard text books:~\cite{tata}, 
\cite{oda}, or \cite{fulton}. 

\begin{notation}
We will work over some fixed field $k$ throughout this 
note. 
Let $X$ be a complete toric variety; 
a $1$-cycle of $X$ is a formal sum $\sum a_iC_i$ with 
complete curves $C_i$ on $X$, and 
$a_i\in \mathbb Z$. 
We put 
$$
Z_1(X):=\{1\text{-cycles of} \ X\}, 
$$
and 
$$
Z_1(X)_{\mathbb R}:=
Z_1(X)\otimes \mathbb R.  
$$
There is a pairing 
$$
\Pic (X)\times Z_1(X)_{\mathbb R}
\to \mathbb R
$$
defined by 
$(\mathcal L, C)\mapsto \deg _C\mathcal L$, extended 
by bilinearity. 
Define 
$$
N^1(X):=(\Pic (X)\otimes \mathbb R)/\equiv
$$
and 
$$
N_1(X):= Z_1(X)_{\mathbb R}/\equiv, 
$$ 
where the {\em numerical equivalence} $\equiv$ 
is by definition the smallest equivalence relation which 
makes $N^1$ and $N_1$ into dual 
spaces. 

Inside $N_1(X)$ there is a distinguished cone of effective 
$1$-cycles, 
$$
{\NE}(X)=\{\, Z\, | \ Z\equiv 
\sum a_iC_i \ \text{with}\ a_i\in \mathbb R_{\geq 0}\}
\subset N_1(X).  
$$ 
It is known that $\NE(X)$ is a rational polyhedral cone. 
A subcone $F\subset {\NE}(X)$ is said to be {\em{extremal}} 
if $u,v\in {\NE}(X)$, $u+v\in F$ imply $u,v\in F$. 
The cone $F$ is also called an 
{\em{extremal face}} of ${\NE}(X)$. 
A one-dimensional extremal face is called an {\em{extremal 
ray}}. 

We define the {\em{Picard number}} $\rho(X)$ by
$$
\rho (X):=\dim _{\mathbb R}N^1(X)< \infty. 
$$
An element $D\in N^1(X)$ is called {\em{nef}} if 
$D\geq 0$ on ${\NE}(X)$. 

We define the {\em{nef cone}} $\Nef (X)$, the
{\em{ample cone}} $\Amp(X)$, and 
the
{\em{pseudo-effective cone}} $\PE(X)$ in $N^1(X)$ as follows.
$$
\Nef(X)=\{D\, |\,  D \text{\   is nef}\},
$$
$$
\Amp(X)=\{D\, |\,  D\text{\ is ample}\}
$$
and
$$
\PE(X)=
$$
$$\{D\equiv \sum a_i D_i\,|\, 
\text{$D_i$ is an 
effective Weil divisor and $a_i\in \mathbb R_{\geq 0}$}\}.
$$
It is not difficult to see that $\PE(X)$ is a rational polyhedral cone in
$N^1(X)$ since $X$ is toric. 
For the usual definition of $\PE(X)$, see, for 
example, \cite[Definition 2.2.25]{lazarsfeld}. 
It is easy to see that $\Amp(X)\subset \Nef (X)\subset \PE(X)$.

From now on, we assume that $X$ is projective. 
Let $D$ be an $\mathbb R$-Cartier 
divisor on $X$. 
Then $D$ is called {\em{big}} if $D\equiv A+E$ for an ample 
$\mathbb R$-divisor $A$ and 
an effective $\mathbb R$-divisor $E$. 
For the original definition of a big divisor, 
see, for example, \cite[2.2 Big Line Bundles and 
Divisors]{lazarsfeld}. We define 
the {\em{big cone}} $\xBig (X)$ in $N^1(X)$ 
as follows. 
$$
\xBig (X)=\{ D\, |\, D \, \text{is big}\}. 
$$ 
It is well known that 
the big cone is the interior 
of the pseudo-effective 
cone and the pseudo-effective cone is the closure 
of the big cone. 
See, for example, \cite[Theorem 2.2.26]{lazarsfeld}. 

\end{notation}

In \cite{kleiman} and \cite{smooth}, 
we mainly treated {\em{non-projective}} toric 
varieties. 
In this note, we are interested in {\em{projective}} toric 
varieties. 

\section{Supplements to the toric Mori theory}\label{sec2}

We introduce the following new notion. 
It will not be useful when we construct various examples 
of {\em{smooth}} projective toric varieties 
in Section \ref{sec4}. However, we contain it here for the 
future usage. 
By the simple observations in this section, 
we know that the great mass of 
complete toric varieties 
have no nontrivial non-big nef line bundles. 

\begin{defn}\label{11}
Let $X$ be a complete toric variety with 
$\dim X=n$. 
Let $\Delta$ be the fan corresponding to 
$X$. 
Let $G(\Delta)=\{v_1, \cdots, v_m\}$ be the 
set of all primitive 
vectors spanning one dimensional cones in $\Delta$. 
If there exists a relation 
$$
a_{i_1}v_{i_1}+\cdots +a_{i_k}v_{i_k}=0 
$$
such that $\{i_1, \cdots, i_k\}\subset \{1, \cdots, m\}$, 
$a_{i_j}\in \mathbb Z_{>0}$ for 
any $1\leq j\leq k$ with 
$k\leq n$, 
then $X$ is called {\em{\lq special\rq}}. 
If $X$ is not \lq special\rq, 
then $X$ is called {\em{\lq general\rq}}. 
\end{defn}

\begin{ex}
The projective space $\mathbb P^n$ is 
\lq general\rq \  in the sense of 
Definition \ref{11}. 
\end{ex}

Let us prepare the following easy but useful lemmas for 
the toric Mori theory. 
The proofs are obvious. So, we omit them. 

\begin{lem}\label{14}
Let $X$ be a complete toric variety and 
let $\pi: \widetilde X\to X$ be a small projective 
toric $\mathbb Q$-factorialization $($cf.~\cite[Corollary 5.9]{fujino}$)$. 
Assume that $X$ is \lq general\rq\  $($resp.~\lq special\rq$)$. 
Then $\widetilde X$ is also \lq general\rq\  $($resp.~\lq special\rq$)$. 
\end{lem}

More generally, we have the following lemma. 

\begin{lem}\label{15}
Let $X$ and $X'$ be complete toric varieties and 
let $\varphi :X\dashrightarrow X'$ be a proper birational 
toric map. 
Assume that $\varphi$ is an isomorphism in codimension one. 
Then $X$ is \lq general\rq\  if and only if 
so is $X'$. 
\end{lem}

\begin{lem}\label{16}
Let $X$ and $Z$ be a complete toric varieties and 
let $\pi:X\to Z$ be a birational toric morphism. 
Assume that $X$ is \lq general\rq. 
Then $Z$ is \lq general\rq. 
We note that $Z$ is not necessarily 
\lq special\rq\  even if $X$ is \lq special\rq. 
\end{lem}

We have two elementary properties. 

\begin{prop}\label{18}
Let $X$ be a complete 
toric variety and let $f:X\to Y$ be a proper surjective 
toric morphism onto $Y$. 
Assume that $X$ is \lq general\rq\  and 
that $\dim Y<\dim X$. 
Then $Y$ is a point. 
\end{prop}

\begin{proof} 
It is obvious. 
\end{proof}

\begin{cor}\label{19}
Let $X$ be a complete toric variety. 
Assume that $X$ is \lq general\rq. 
Let $D$ be a nef Cartier divisor on $X$ such that 
$D\not \sim 0$. 
Then $D$ is big. 
\end{cor}

\begin{proof}
Since $D$ is nef, the linear system $|D|$ defines 
a proper surjective toric morphism $\Phi_{|D|}: X\to Z$. 
Apply Proposition \ref{18} to $\Phi_{|D|}:X\to Z$. 
Then we obtain $\dim Z=\dim X$. 
Therefore, $D$ is big.  
\end{proof}

The next proposition is also obvious. We contain it 
for the reader's convenience because 
it has not been stated explicitly in the 
literature. 
For the details of the toric Mori theory, 
see \cite[Section 5]{fujino} and 
\cite{intro}. 

\begin{prop}[MMP for \lq general\rq \ projective toric varieties]\label{17}
Let $X$ be a $\mathbb Q$-factorial 
projective toric variety and let $B$ be a 
Cartier divisor on $X$ such 
that $B$ is not pseudo-effecitve. 
Assume that $X$ is \lq general\rq. 
We run the MMP with respect to 
$B$. 
Then we obtain a sequence of 
$B$-negative divisorial contractions and 
$B$-flips$:$ 
$$
X=X_0\dashrightarrow X_1\dashrightarrow 
\cdots \dashrightarrow X_i\dashrightarrow 
X_{i+1}\dashrightarrow \cdots 
\dashrightarrow X_l, 
$$ 
where $X_l$ is a $\mathbb Q$-factorial 
projective toric variety with $\rho (X_l)=1$. 
\end{prop}

\begin{proof}
Run the MMP with respect to $B$, where 
$B$ is not pseudo-effective, for example, 
$B=K_X$. 
Since $B$ is not pseudo-effective, the 
final step is a Fano contraction $X_l\to Z$. 
Since $X$ is \lq general\rq, 
$X_l$ is also \lq general\rq \ by Lemmas \ref{15} and \ref{16}. 
Therefore, 
$Z$ must be a point by Corollary \ref{18}. 
This means that 
$X_l$ is a $\mathbb Q$-factorial 
projective toric variety with $\rho(X_l)=1$. 
\end{proof}

We will see that any smooth projective 
toric variety $X$, which is not isomorphic to 
the projective space, is \lq special\rq\  by \cite{batyrev}. 
See Proposition \ref{25} below. 
So, the results in this section can not be applied to 
{\em{smooth}} projective toric varieties. 

\section{Primitive collections and relations}\label{sec3}

Let us recall the notion of primitive collections and 
primitive relations introduced by Batyrev (cf.~\cite{batyrev}). 
It is very useful to compute some explicit examples of 
toric varieties. Note that 
this section is not indispensable 
for understanding the examples in Section \ref{sec4}. 

Let $\Delta$ be a complete non-singular 
$n$-dimensional fan and 
let $G(\Delta)$ be the set of all primitive generators of 
$\Delta$. 
\begin{defn}[Primitive collection]
A non-empty subset $\mathcal P=\{v_1,$ $\cdots,v_k\}\subset 
G(\Delta)$ is called a {\em{primitive collection}} if 
for each generator $v_i\in \mathcal P$ the elements of $\mathcal 
P\setminus \{v_i\}$ generate a $(k-1)$-dimensional 
cone in $\Delta$, while $\mathcal P$ does not 
generate any $k$-dimensional 
cone in $\Delta$. 
\end{defn}
\begin{defn}[Focus]
Let $\mathcal P=\{v_1, \cdots, v_k\}$ be a 
primitive collection in $G(\Delta)$. 
Let $S(\mathcal P)$ denote $v_1+\cdots +v_k$. 
The {\em{focus}} $\sigma (\mathcal P)$ of $\mathcal P$ 
is the cone in $\Delta$ of the smallest 
dimension containing $S(\mathcal P)$.  
\end{defn}
\begin{defn}[Primitive relation]
Let $\mathcal P=\{v_1, \cdots, v_k\}$ be a primitive 
collection in $G(\Delta)$ and $\sigma(\mathcal P)$ 
its focus. 
Let $w_1, \cdots, w_m$ be the primitive generators 
of $\sigma (\mathcal P)$. Then 
there exists a unique linear combination $a_1w_1+\cdots 
+a_mw_m$ with positive integer coefficients 
$a_i$ which is equal to $v_1+\cdots+ v_k$. Then the linear 
relation $v_1+\cdots +v_k-a_1w_1-\cdots -a_mw_m=0$ 
is called the {\em{primitive 
relation associated with $\mathcal P$}}. 
\end{defn}

Then we have the description of $\NE(X)$ by primitive 
relations. 

\begin{thm}[{cf.~\cite[2.15 Theorem]{batyrev}}]
Let $\Delta$ be a projective 
non-singular fan and $X=X(\Delta)$ 
the corresponding toric variety. 
Then the Kleiman-Mori cone $\NE(X)$ is generated 
by all primitive relations. 
The primitive relation which spans an extremal 
ray of $\NE(X)$ is said to be {\em{extremal}}. 
\end{thm}

Let $\Delta$ be a {\em{projective}} non-singular 
$n$-dimensional 
fan. 
Then, Batyrev obtained the following important 
result. 

\begin{prop}[{cf.~\cite[3.2 Proposition]{batyrev}}]\label{25}
There exists a primitive collection $\mathcal P=\{v_1, \cdots, 
v_k\}$ in $G(\Delta)$ such that 
the associated primitive 
relation is of the form 
$$
v_1+\cdots +v_k=0. 
$$
In other words, the focus $\sigma(\mathcal P)=\{0\}$. 
\end{prop}

We close this section with an elementary remark. 

\begin{rem}
If $k=n+1$ in Proposition \ref{25}, 
then $X(\Delta)\simeq \mathbb P^n$. 
\end{rem}

Therefore, a smooth projective 
toric variety $X$ is \lq general\rq\ if and only if 
$X$ is isomorphic to the projective 
space. 
By this reason, it is not so easy to 
construct smooth projective toric varieties 
on which any nontrivial 
nef line bundles are big. 

\section{Examples}\label{sec4}
First, let us recall the following example, 
which is not a toric variety. 
For the details, see \cite{morimukai} and 
\cite[p.~67]{matsuki}. 

\begin{ex}[{\cite[no.~30 Table 2]{morimukai}}]\label{41}
Let $X$ be the blowing-up of $\mathbb P^3_{\mathbb C}$ along 
a smooth conic. 
Then $X$ is a smooth Fano threefold with $\rho (X)=2$. 
It is known that $X$ has two extremal divisorial 
contractions. 
One contraction is the inverse of the blowing-up $X\to \mathbb P^3$. 
Another one is a contraction of $\mathbb P^2$ on 
$X$ into a smooth point. 
Therefore, it is not difficult to see that every nef Cartier divisor 
$D\not \equiv 0$ is big. 
\end{ex}

The next example is the main theme of this short note. 
It is hard for the non-experts to find it. 
Therefore, we think it is worthwhile to describe it 
explicitly here. 
 
\begin{ex}\label{321}
We put $v_1=(1,0,0), v_2=(0,1,0), 
v_3=(0,0,1)$, and $v_4=(-1, -1, -1)$. 
We consider the standard fan of $\mathbb P^3$ generated 
by $v_1, v_2, v_3$, and $v_4$. 
We subdivide the cone $\langle v_1, v_2, v_4\rangle$ as 
follows. Take a blow-up $X_1\to \mathbb P^3$ along the 
vector $v_5=(1, -1, -2)=3v_1+v_2+2v_4$. 
We take a blow-up $X_2\to X_1$ along the 
vector $v_6=(1, 0, -1)=\frac{1}{2}(v_1+v_2+v_5)$ and 
a blow-up $X_3\to X_2$ along 
$v_7=(0, -1, -2)=\frac{1}{3}(v_2+2v_4+2v_5)$. 
Finally, we take a blow-up $X_3$ along the vector $v_8=(0, 0, -1)
=\frac{1}{2}(v_2+v_7)$ and obtain $X$. 
Then, it is obvious that $X$ is projective 
and $\rho (X)=5$. 
It is easy to see that $X$ is smooth. 
In this case, $\NE(X)$ is spanned by 
the following five extremal primitive relations, 
$v_1+v_2+v_5-2v_6=0$, 
$v_4+v_5+v_8-2v_7=0$, 
$v_2+v_7-2v_8=0$, 
$v_6+v_8-v_2-v_5=0$, 
and $v_3+v_5-2v_1-v_4=0$. 
This toric variety $X$ is nothing but the 
one labeled as [8-10] in \cite[Theorem 9.6]{tata}. 
The picture below helps us understand the combinatorial 
data of $X$. 

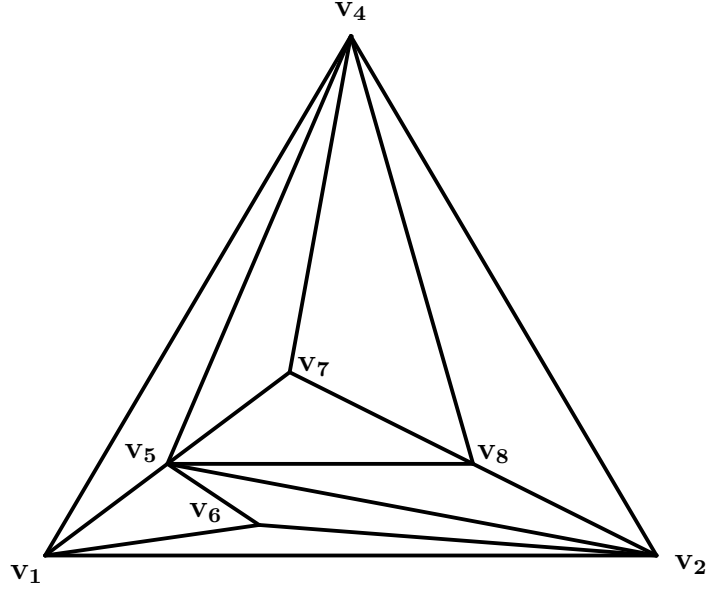
\begin{center}
\scalebox{1.0}{
\unitlength 0.1in
\begin{picture}( 36.1700, 29.4800)( 14.1000,-31.0600)
%
\special{pn 20}%
\special{pa 3332 376}%
\special{pa 3332 376}%
\special{fp}%
\special{pa 3332 376}%
\special{pa 1732 3096}%
\special{fp}%
%
\special{pn 20}%
\special{pa 1732 3096}%
\special{pa 4932 3096}%
\special{fp}%
%
\special{pn 20}%
\special{pa 4932 3096}%
\special{pa 3332 376}%
\special{fp}%
%
\special{pn 20}%
\special{pa 3012 2136}%
\special{pa 3332 376}%
\special{fp}%
%
\special{pn 20}%
\special{pa 3012 2136}%
\special{pa 1732 3096}%
\special{fp}%
%
\special{pn 20}%
\special{pa 3012 2136}%
\special{pa 4932 3096}%
\special{fp}%
%
\special{pn 20}%
\special{pa 2372 2616}%
\special{pa 3332 376}%
\special{fp}%
%
\special{pn 20}%
\special{pa 2372 2616}%
\special{pa 3972 2616}%
\special{fp}%
%
\special{pn 20}%
\special{pa 3972 2616}%
\special{pa 3332 376}%
\special{fp}%
%
\special{pn 20}%
\special{pa 2372 2616}%
\special{pa 4932 3096}%
\special{fp}%
%
\special{pn 20}%
\special{pa 2852 2936}%
\special{pa 2372 2616}%
\special{fp}%
%
\special{pn 20}%
\special{pa 2852 2936}%
\special{pa 1732 3096}%
\special{fp}%
%
\special{pn 20}%
\special{pa 2852 2936}%
\special{pa 4932 3096}%
\special{fp}%
\put(50.2700,-31.9100){\makebox(0,0)[lb]{$\mathbf{v_2}$}}%
\put(22.3500,-25.5900){\makebox(0,0){$\mathbf{v_5}$}}%
\put(26.5900,-29.1900){\makebox(0,0)[rb]{$\mathbf{v_6}$}}%
\put(16.3500,-31.9100){\makebox(0,0){$\mathbf{v_1}$}}%
\put(33.3300,-2.4300){\makebox(0,0){$\mathbf{v_4}$}}%
\put(39.9300,-26.1300){\makebox(0,0)[lb]{$\mathbf{v_8}$}}%
\put(30.5300,-21.5300){\makebox(0,0)[lb]{$\mathbf{v_7}$}}%
\end{picture}%
}
\end{center}
\begin{claim}
There are no projective surjective toric 
morphism $f:X\to Y$ with 
$\dim Y=1$ or $2$. 
\end{claim}
\begin{proof}
The variety $X$ is obtained by successive 
blowing-ups of $\mathbb P^3$ inside 
the cone $\langle v_1, v_2, v_4\rangle$. 
So, $X$ does not admit to a morphism to a curve. 
Thus, we have to consider the case when $Y$ is a surface. 
By considering primitive relations, $f:X\to Y$ must be 
induced by the projection $\mathbb Z^3\to \mathbb Z^2: 
(x, y, z)\mapsto (x, y)$ because 
$v_3+v_8=0$. The image of the cone 
$\langle v_2, v_5, v_8\rangle$ is the 
cone spanned by $(0, 1)$ and $(1, -1)$. 
On the other hand, the image of the cone 
$\langle v_1, v_4, v_5\rangle$ is the cone 
spanned by $(1, 0)$ and $(-1, -1)$. 
Therefore, there are no surjective morphisms 
$f:X\to Y$ with $\dim Y=2$. 
\end{proof}
Thus, every nef divisor $D\not \sim 0$ is big, 
that is, $\partial \Nef (X)\cap \partial \PE(X)
=\{0\}$. 
\end{ex} 

By the following example, the reader 
understands the advantage of using the toric geometry to 
construct examples. 
We do not know what happens if we take blow-ups 
of $X$ in Example \ref{41}. 

\begin{ex}\label{43}
By taking blowing-ups inside the cone $\langle v_5, v_7, v_8\rangle$ in Example \ref{321}, we obtain a smooth projective toric 
threefold $X_k$ for any $k\geq 6$ such that 
$\rho (X_k)=k$ and 
$\partial \Nef (X_k)\cap \partial \PE(X_k)=\{0\}$, that is, 
every nef divisor $D\not \sim 0$ on $X_k$ is big. 
More explicitly, for example, $X_6$ is the 
blow-up of $X$ along $u_6=v_5+v_7+v_8$ and 
$X_{k+1}$ is the blow-up of $X_k$ along 
$u_{k+1}=v_5+v_7+u_k$ for $k\geq 6$. 
\end{ex}

We can easily check that any smooth projective 
toric 
threefolds of Picard number $2\leq \rho \leq 4$ 
have some nontrivial non-big nef line bundles by the 
classification table in \cite[Theorem 9.6]{tata}. 
For smooth non-projective toric variety, the following 
example will help the reader. It is the most famous example of 
smooth complete non-projective toric 
threefold. 

\begin{ex}
Let $\Delta$ be the fan whose 
rays are spanned by 
$v_1=(1, 0, 0)$, $v_2=(0, 1, 0)$, 
$v_3=(0, 0, 1)$, $v_4=(-1, -1, -1)$, 
$v_5=(0, -1, -1)$, $v_6=(-1, 0, -1)$, 
$v_7=(-1, -1, 0)$, and 
whose maximal cones are 
$\langle v_1, v_2, v_3\rangle$, 
$\langle v_4, v_5, v_6\rangle$, 
$\langle v_4, v_6, v_7\rangle$, 
$\langle v_4, v_5, v_7\rangle$, 
$\langle v_1, v_2, v_5\rangle$, 
$\langle v_2, v_5, v_6\rangle$, 
$\langle v_2, v_3, v_6\rangle$, 
$\langle v_3, v_6, v_7\rangle$, 
$\langle v_1, v_3, v_7\rangle$, 
$\langle v_1, v_5, v_7\rangle$. 
Then $X=X(\Delta)$ is the most famous non-projective 
smooth toric threefold with $\rho (X)=4$ obtained by 
Miyake and Oda. 
By removing three two-dimensional 
walls $\langle v_1, v_7\rangle$, $\langle v_2, v_5\rangle$, 
and $\langle v_3, v_6\rangle$ from 
$\Delta$, we obtain a flopping contraction 
$f:X\to Y$. 
It is easy to see that 
$Y$ is a projective toric threefold 
with $\rho (Y)=2$ and three ordinary double 
points. 
We can check that every nef divisor 
$D$ can be written as $D=f^*D'$ for some nef 
divisor $D'$ on $Y$. 
On the other hand, $\Nef (Y)$ is a two dimensional 
cone and every nef divisor on $Y$ is big. 
Therefore, $\Nef (X)$ is also 
two-dimensional and all the 
nef divisors on $X$ are big. 
We note that $\Nef (X)$ is thin in $N^1(X)$ by Kleiman's 
ampleness criterion since $X$ is a smooth complete 
non-projective variety. 
\end{ex}

The reader can find many smooth complete 
non-projective toric threefolds $X$ with $\Nef (X)=\{0\}$ 
in \cite{smooth}. 

\section{Miscellaneous comments}\label{sec5}
In this final section, we collect miscellaneous results. 
First, we explain how to generalize Examples \ref{321} and 
\ref{43} in dimension $\geq 4$. 

\begin{say}We put 
$v_1=(1, 0, \cdots, 0)$, 
$v_2=(0, 1, 0, \cdots, 0)$, 
$v_3=(0, 0, 1, 0, \cdots, 0)$, 
$v_4=(-1, -1, \cdots, -1)\in N=\mathbb Z^n$. 
We consider $w_1=(0, 0, 0, 1, 0, \cdots, 0)$, 
$w_2=(0, 0, 0, 0, 1, 0, \cdots, 0)$, 
$\cdots$, $w_{n-3}=(0, \cdots, 0, 1)\in N$. 
By these vectors, we can construct a fan corresponding 
to $\mathbb P^n$ as usual. 
We take $v_5=3v_1+v_2+2v_4=
(1, -1, -2, \cdots, -2)$, 
$v_6=\frac{1}{2}(v_1+v_2+v_5)
=(1, 0, -1, \cdots, -1)$, 
$v_7=\frac{1}{3}(v_2+2v_4+2v_5)
=(0, -1, -2, \cdots, -2)$, 
and $v_8 =\frac{1}{2}(v_2+v_7)=(0, 0, -1, \cdots, -1)$. 
We take a sequence of blow-ups 
$$
X\to X_3\to X_2\to X_1\to \mathbb P^n
$$ 
as in Examples \ref{321}. 
In this case, the center of each blow-up 
is $(n-3)$-dimensional. We can 
easily check that 
$X$ is a smooth projective toric $n$-fold. 
We note that $v_3+w_1+\cdots+w_{n-3}+v_8=0$. 
\begin{claim}
If $f: X\to Y$ is a proper surjective toric morphism 
and $Y$ is not a point, then $\dim Y=n$. 
\end{claim}
\begin{proof}[Proof of {\em{Claim}}]
By considering linear 
relations among $v_1, v_2, \cdots, 
v_8, w_1,$ $\cdots, w_{n-3}$ 
as in Definition \ref{11}, 
$f$ should be induced by the projection 
$\mathbb Z^n\to \mathbb Z^2: (x_1, x_2, \cdots, 
x_n)\mapsto (x_1, x_2)$ if 
$\dim Y<n$. 
By the same arguments as in the proof of 
Claim in Example \ref{321}, 
it can not happen. 
Therefore, we obtain 
$\dim Y=n$. 
\end{proof}
Thus, any nontrivial 
nef line bundles on $X$ are 
big. 
\end{say}
So, for any $(n, \rho)$, where 
$n\geq 4$ and $\rho \geq 5$, 
we can construct a 
smooth projective 
toric $n$-fold $X$ with $\rho (X)=\rho$ on which 
any nontrivial nef line bundles are big (cf.~Example \ref{43}). 
We leave the details for the reader's exercise. 
The next one is a higher dimensional 
analogue of \cite{smooth}. 

\begin{say}[Smooth complete toric varieties with no nontrivial nef line bundles]
Let $X$ be a smooth complete toric variety 
with no nontrivial 
nef line bundles. We put 
$\mathcal E=\mathcal O_X^{\oplus k}\oplus \mathcal L$ 
for $k\geq 1$, where 
$\mathcal L$ is a nontrivial line bundle on $X$. 
We consider the $\mathbb P^k$-bundle 
$\pi:Y=\mathbb P_X(\mathcal E)\to X$. 
Then $Y$ is a $(\dim X+k)$-dimensional 
complete toric variety. 
It is easy to 
see that there are no nontrivial nef line bundles on $Y$. 
So, for $n\geq 4$, we 
can construct many $n$-dimensional 
smooth complete toric varieties of Picard number $\geq 6$ 
with no nontrivial nef line bundles by \cite{smooth}. 
\end{say}

Finally, we close this note with an easy result. 
We treat the other extreme case:~$\Nef (X)=\PE(X)$. 

\begin{prop}
Let $X$ be a $\mathbb Q$-factorial 
projective toric variety with 
$\rho (X)=\rho$. 
Assume that $\Nef (X)=\PE(X)$, that is, 
every effective divisor is nef. 
Then there is a finite toric morphism 
$\mathbb P^{n_1}\times \cdots \times 
\mathbb P^{n_\rho}\to X$ with 
$n_1+\cdots+ n_{\rho}=\dim X$. 
When $X$ is smooth, 
$X\simeq \mathbb P^{n_1}\times \cdots \times \mathbb P^{n_\rho}$ with $n_1+\cdots+ n_{\rho}=\dim X$. 
\end{prop}

\begin{proof}
The condition $\Nef (X)=\PE(X)$ implies that 
every extremal ray of $\NE(X)$ is a Fano type. 

First, we assume that 
$X$ is smooth. 
We obtain a Fano contraction $f:X\to Y$ with 
$\rho(Y)=\rho (X)-1$, where 
$Y$ is a smooth projective 
toric variety and $\Nef (X)=\PE(X)$. 
It is well known that $X$ is a projective 
space bundle over $Y$. 
By the induction, we obtain $Y\simeq \mathbb P^{n_1}
\times \cdots \times\mathbb P^{n_{\rho-1}}$. Therefore, 
we can easily check that 
$X\simeq \mathbb P^{n_1}\times \cdots \times
\mathbb P^{n_\rho}$ and 
$f$ is the projection. 
Lemma \ref{key} may help the reader check it. 

Next, we just assume that $X$ is a $\mathbb Q$-factorial 
projective toric variety with $\Nef(X)=\PE(X)$. 
As above, we have a Fano contraction $f:X\to Y$ with 
$\rho (Y)=\rho (X)-1$. 
In this case, $Y$ is a $\mathbb Q$-factorial 
projective toric variety with $\Nef(Y)=\PE(Y)$. By applying 
the induction, we have a finite toric surjective 
morphism 
$g:W'=\mathbb P^{n_1}\times \cdots \times 
\mathbb P^{n_{\rho-1}}\to Y$. If we need, we 
take a higher model $W=\mathbb P^{n_1}\times \cdots
\times \mathbb P^{n_{\rho-1}}\to W'\to Y$ and can assume that 
$V\to W$ is a fiber bundle, where $V$ is the normalization 
of $W\times _YX$. 
We note that $\Nef(V)=\PE(V)$. 
For any irreducible 
torus invariant closed subvariety 
$U$ on $V$ such that 
$\dim U=\dim W+1$ and 
that $U\to W$ is surjective, we can see 
that $U$ is 
a $\mathbb P^1$-bundle over 
$W$ and $\Nef(U)=\PE(U)$. 
Therefore, 
$U\simeq W\times \mathbb P^1$ and $U\to W$ is 
the first projection by the previous step. 
By these observations, we can see that $V\simeq W\times F$, where $F$ is a $\mathbb Q$-factorial projective 
toric variety with $\rho (F)=1$. 
Thus, we obtain a desired finite toric morphism 
$\mathbb P^{n_1}\times \cdots \times \mathbb P^{n_\rho}\to 
X$. 
\end{proof}

The following property is a key lemma. 
 
\begin{lem}\label{key}
Let $X$ be a $\mathbb Q$-factorial 
projective 
toric variety with $\Nef (X)=\PE(X)$. 
Let $Z$ be any irreducible 
torus invariant closed subvariety of $X$. Then $Z$ is a 
$\mathbb Q$-factorial projective toric variety with 
$\Nef (Z)=\PE(Z)$. 
\end{lem}
\begin{proof}
It is obvious. 
\end{proof} 
 
\begin{ack}
We would like to thank Professor Noboru Nakayama for 
some comments. 
The first author 
was partially supported by The Inamori Foundation and by 
the Grant-in-Aid for Young Scientists (A) $\sharp$20684001 from 
JSPS. 
The second author was supported by the Grant-in-Aid for Young 
Scientists (B) $\sharp$20740026 from JSPS. 
\end{ack}
\ifx\undefined\bysame
\newcommand{\bysame|{leavemode\hbox to3em{\hrulefill}\,}
\fi

\end{document}